\documentclass{article}

\usepackage[T1]{fontenc}
\usepackage{amsmath,amsthm,amsfonts,amssymb}
\usepackage[mathscr]{eucal}
\usepackage{indentfirst}
\usepackage{graphicx}
\usepackage{exscale,relsize}
\usepackage{multicol}
\usepackage[dvips,final]{epsfig}

\theoremstyle{plain}
\newtheorem{teo}{Theorem}[section]
\newtheorem{cor}{Corollary}[section]
\newtheorem{prop}{Proposition}[section]
\newtheorem{rem}{Remark}[section]
\newtheorem{lem}{Lemma}[section]
\newtheorem{defn}{Definition}[section]
\newtheorem{ex}{Example}[section]

\theoremstyle{definition}

\def\N{{\mathbb N}}
\def\qed{\hfill $\square$\vskip.5cm}
\def\P{\mathbb{P}}
\def\C{{\mathcal C}}

\def\ind{{\bf 1\kern-1.2mm1}}
\def\Zset{Z\kern-0.45emZ}
\def\Rset{I\kern-0.3emR}
\def\Nset{I\kern-0.3emN}
\def\E{I\kern-0.3emE}
\def\V{\kern-0.3emVar}

\def\lp{ \left( }
\def\rp{ \right) }
\def\ll{ \left\{ }
\def\rl{ \right\} }
\def\lc{ \left[ }
\def\rc{ \right] }
\def\lv{ \left\vert }
\def\rv{ \right\vert }

\def\go{ \rightarrow }

\def\lf{\lfloor}
\def\rf{\rfloor}
\def\lr{\lceil}
\def\rr{\rceil}

\def\vs{\vskip.4cm}
\def\beq{ \vs\begin{displaymath} }
\def\eeq{ \end{displaymath}\vs\noindent  }
\def\beqn{ \vs\begin{equation} }
\def\eeqn{ \end{equation}\vs\noindent  }
\def\beqa{ \vs\begin{eqnarray*} }
\def\eeqa{ \end{eqnarray*}\vs\noindent }
\def\beqan{ \vs\begin{eqnarray} }
\def\eeqan{ \end{eqnarray}\vs\noindent }
\def\nn{ \nonumber }

\def\pmix{ $\psi$-mixing }

\def\Nset{I\kern-0.3emN}
\def\Zset{{Z\kern-0.6emZ}}
\def\Rset{I\kern-0.3emR} 
\def\bbe{I\kern-0.3emE}
\def\one{1\kern-0.3em1}

\def\L{\mathcal{L}}
\def\wn{x_{1}^{n}}

\title{The distribution of the overlapping function}
\author{Miguel Abadi
\thanks{Instituto de Matem\'{a}tica e Estat\'{i}stica, Universidade de S\~{a}o Paulo.}
\and Rodrigo Lambert {\small{*}}
\date{}
}

\begin{document}
\maketitle

\begin{abstract}
We consider the set of finite sequences of length $n$ over a finite or countable alphabet $\C$.
We consider the function defined over $\C^n$ which gives the size of the maximum overlap of a 
given sequence with a (shifted) copy of itself.
We compute the exact distribution and
the limiting distribution of this function when the sequence is chosen according to a product measure
with marginals identically distributed.
We give a point-wise upper bound for the velocity of this convergence.
Our results holds for a finite or countable alphabet.
The non-parametric distribution is related to the prime decomposition of positive integers.
We illustrate with some examples.
\end{abstract}
{\bf Running head} The distribution of the overlapping function\\
{\bf Subject class} 60Axx, 60C05, 60-XX, 60Fxx, 41A25

\section{Introduction}

Consider a positive integer $n$.
Consider the space of all sequences of length $n$ defined over a finite or countable alphabet $\C$.
In this work we consider the  function $S_n$ defined over ${\mathcal C}^n$ and taken values on $\{0,\dots,n-1\}$.
For each string, this function gives the size of the maximum overlap  of the string with a 
(shifted) copy of itself and zero if there is no overlap. See Definition \ref{overlap}. 

The function $S_n$ is related to the \emph{first return} function $T_n$ that gives the \emph{minimum number of shifts} 
we have to apply to the sequence in order to
find an overlap with a copy of itself through the formula $S_n=n-T_n$.

The relevance of the first return function (and consequently of the overlapping function) 
was put in evidence in the statistical 
analysis of the Poincare recurrence.
To prove convergence of the number of occurrences of a string  (say of length $n$) 
as $n$ diverges, to the Poisson distribution it is necessary that the string  does not overlap itself
\cite{Hir}. 
Or at least, that the proportion that overlaps, with respect to $n$ is small
\cite{AbHitting, AbVeReturn}. 
If this is not the case, a compound Poisson distribution is the limiting law \cite{HaVaCompound}.
There are also some approximations for this limit \cite{ReSc99, ReSc98, RoSc}.

It also appears when we consider the time elapsed until the first occurrence of the string.
This time is called the \emph{hitting time}.
It is known that the hitting time can be well approximated by an exponential law with 
parameter given by the measure of the string \cite{HSV}.
But when the string overlaps itself, the parameter must be corrected by a factor which is
the probability that the string does not appear twice consecutively.
And this probability is given by the overlapping properties of the string
\cite{aba1, aba4, AbSa, GS}.

Yet,  it appears when consider the \emph{return time} instead of the hitting time. 
This law can  also be well approximated by an exponential law with parameter
being the measure of the string \cite{HSV}.
However when the string overlaps itself the limiting law is a convex combination of a
Dirac measure at the origin and an exponential law \cite{aba4, AbSa, AbVeReturn}.
As in the case of the hitting time, the parameter must be corrected by the 
above given factor. The weight of the convex combination is again this parameter.
Surprisingly, when taking expectation (but not any other moment \cite{aba4}) this parameter cancels.
This fact is hidden when looking at Kac's Lemma (\cite{Kac}).

As far as we know, the first paper to notice that the measure of all strings that have large overlaps 
converges to zero was \cite{CoGaSc}. The authors proved the exponential decay of this measure
when "large" means larger  or equal than $2n/3$. That result holds for $\psi$-mixing processes 
with exponentially decaying function $\psi$ and with finite alphabet.
Later, the same was generalized  in \cite{aba1} to $\phi$-mixing processes.
Here, "large" means larger or equal than a certain proportion $Cn$ where $C$ is a constant depending on 
the cardinal of the alphabet.

Let us denote with  $T_n(x_1^n)=n-S_n(x_1^n)$ the number of shifts needed
to get the first overlap of an
$n$-string $x_1^n=(x_1,\dots,x_n)$ with itself.
It was proved in \cite{STV} using Kolmogorov complexity function and independently in
\cite{ACS} using Shannon, Mc-Millan \& Breiman's Theorem that for a stochastic process over a finite alphabet, and with an ergodic  measure $\mu$ with  positive metric entropy satisfying the specification property \cite{KH},
the ratio $T_n/n$ verifies 
\beq
\liminf_{n\go \infty} \frac{T_n(x_1^n)}{n} = 1   \ , 
\eeq
for almost every sequence $x=(x_1,x_2,\dots)$.

This result has also been proved for a
class of non-uniformly expanding maps of the interval \cite{HSV} in the context of dynamical systems.

Even when the definition of $S_n$ (and $T_n$) are purely combinatorial, it is interesting to have in mind
an equivalent definition from the dynamical point of view.
Fixed an $n$-string $\wn$, the return time of $\wn$ over all infinite sequence $y_1^{\infty}$, such that $y_1^n = \wn$ (i.e., a cilinder indexed by $\wn$),
is defined explicitely as
\[
\tau_{\wn}( y_1^{\infty} )  = \inf\{ t \ge 2 \ | \ y_{t+1}^{t+n}=\wn \} \ ,
\]
(and infinite otherwise).
Then 
\beqn \label{dyn}
T_n(\wn) = \inf_{ x_{n+1}^{\infty} } \tau_{\wn}(x_{n+1}^{\infty}) \ .
\eeqn
Namely, the first return (of the finite sequence $x_1^n$) function $T_n(\wn)$ is the infimum 
of the return time of $\wn$ over all the realizations of the process $x_1^\infty$
that have  as initial condition $\wn$. 
Thus, $T_n$ is called the first return of (the $n$-string) $\wn$ in the dynamical literature.

A  large deviation principle for $T_n$  was succesively proved in \cite{AbCa, AbVa, HaVa}
for processes that verify different types of mixing conditions, including 
product measures, ergodic Markov chains, Gibbs measures with Holder continuous potential, etc
.
The limit of the deviation function is related to the Renyi entropy of the measure that generates the strings
(see, for instance, \cite{zyc} for definition and properties of the Renyi entropy). 
The existence of the Renyi entropies are also proved.

Studying celular automatas, \cite{RoTo} showed that for a counting measure over a finite alphabet, 
the proportion of strings with no overlap converges to a positive constant.

Until now, nothing was known about the distribution 
of $T_n$ and the existence of its limit  reminded unknown.
Since the sequence of random variables $T_n$ are not tight, we are lead to consider instead $S_n=n-T_n$.
In this work we consider  a product measure $\P$ over $\C^n$ with marginals identically distributed.
Namely, the marginal of $\P$ is a probability function  over $\C$, which may be finite or countable.
Thas is, the string $\wn$ are generated by independent, identically distributed random variables.
Each of this random variables has a probability distribution defined by a vector of parameters $\theta=(p_\alpha)_{\alpha\in A}$
lying in the parameter space 
$$
\Theta=\{ \theta=(p_\alpha)_{\alpha\in A}\ | \  p_\alpha \ge 0  \ , 
\  \sum_{\alpha\in A} p_\alpha=1 \}\subset (0,1)^A \ .
$$ 

Our main result read as follows:
We present explicit expressions for the probability mass function of $S_n$
and also for its cumulate distribution $\P(S_n \ge k)$.
Moreover we show their  convergence 
to a non-degenerated limiting distribution.

The limiting probability mass function reads $q_k= m_2^{2k}-b_k$
where $m_2$ is the $\ell_2$-norm of the parametric vector $\theta$, namely 
$\sqrt{\sum_{\alpha\in A} p_\alpha^2}$,
and $b_k$ is a smaller order term.
Thus, the limiting distribution has an exponentially decreasing tail.
We observe that, as in the aforementioned case of the large deviation of $S_n$, 
the probability of $S_n$ is also related to the Renyi entropy function $R_H(\beta)$,
in this case at $\beta=1$.
We also present an explicit expression for the correction term $b_k$.
It is also related to the Renyi entropies, this time at
positive integers $\beta$.
We also show that a similar result holds for the cumulated distribution of $S_n$.
As an application, we show that for the uniform (counting) measure,
the limiting measure of the non-overlapping strings ($S_n=0$)
is related to the prime decomposition of the positive integers.

The dynamical definition of $T_n$ (and therefore of $S_n$) allows us to think
that this random variables are defined in the common space of infinite sequences.
Therefore one may ask about other types of convergences. We 
finish the paper showing that
$S_n$ does not converges in probability to any limiting random variable $S$.



This paper is organized as follows. 
In Section 2 we introduce some notation and the basic definitions. 
In Section 3 we present our results and provide some examples. 
Section 4 presents some tools needed for the proofs.
Section 5 presents the proofs of our theorems.
Finally, Section 6 shows that the convergence in distribution of $S_n$ can not be 
extended to convergence in probability.

\section{Notation and definitions}

We consider a probability product measure with identically distributed mar\-gi\-nals 
over  a finite or countable  alphabet ${\mathcal C}$.
 
The symbols of $\C$ are called letters. The set $\C$ which we index by a set $A$.
We put $p_\alpha,\  \alpha\in A$ for the probability of these letters.
To avoid non-interesting cases we assume that $0<p_\alpha<1$ for all $\alpha$.
Thus, the letters are generated by independent identically distributed random variables.

A finite sequence of consecutive letters of length $n$, is called an \emph{$n$-string} or a \emph{word of length $n$} 
and is denoted with the letter $w$, or $w_i$ or even $w_{i,j}$.
When we need to describe specifically the letters of a finite or infinite sequence, namely
$(x_a,\dots,x_b)$ with $x_i \in \C$ and $0 \le a \le b \le \infty$, we write simply by $x_a^b$.

If $w_i$ is a $n_i$-string, $i=1,2,\cdots k$, with  $n=\sum_{i=1}^{k}n_i$
we write
$w_1w_2....w_n$ for the $n$-string which consists in the concatenation 
of the $n_i$-strings $w_1,w_2,....,w_k$.

The object of our analysis is  the following.

\begin{defn} \label{1overlap}
For a given string $\wn\in\C^n$, the \emph{period} or the  \emph{first return} of $\wn$, denoted by $T_n(\wn)$, 
is defined by the first self-overlapping position of the string.  That is,
$T_n: \C^n \go \{1,\dots,n\}$ with
\begin{equation} \label{periodo}
T_n(\wn)=\min\{k\geq 1 | x_1^{n-k}=x_{k+1}^{n}\},
\end{equation}
and $T_n(\wn)=n$ when the above set is empty. 
\end{defn}

The fact that $T_n/n$ converges to one almost surely implies that $T_n $ is not tight,
therefore it is more convenient to consider the
variables $S_n=n-T_n \in \{0,\dots,n-1\}$.
In this case we have that $S_n/n$ converges to zero almost surely.

\begin{defn}
We define  $S_n(x_1^{n})$  as the \emph{maximum size} of the self-overlap, 
among all the self-overlaps of
the string $\wn$. Namely, 
\[
S_n(x_1^{n})=n-T_n(x_1^{n}) \ .
\]
\end{defn}


\newpage
To study the level sets $\{S_n=k\}$ or even the cumulated sets $\{S_n \le k\}$, with $k\in\{0,\dots,n-1\}$ 
we will use as a tool the following sets.

\begin{defn} \label{blocos}
Let $n$ be a positive integer.
For every positive integer $k< n$, $B_n(k)$ denotes the set of strings $x_1^{n}$ such that 
the first block of length $k$ is whatever it is, but then this block is concatenated
until to complete the $n$ symbols. Namely
$$
x_1^{n} = ( \underbrace{ x_1,\dots,x_k }_{1}, \underbrace{ x_1,\dots,x_k }_{2} ,\dots,
\underbrace{x_1,\dots,x_k }_{\lfloor n/k \rfloor}, \underbrace{x_1,\dots,x_r}_{1} ) \ ,
$$
with $0\leq r < k $.  
If $\lfloor n/k \rfloor = n/k $, wich implies that $r=0$, the last string is the empty string.
\end{defn}

We will also use the following definition.

\begin{defn} \label{overlap}
We set $R_n(k)$ as the set of $n$-strings $x_1^{n} \in \C^n$ such that $x_1^{n}$ has \emph{an} overlap of size $k$.
Namely
\[ R_n(k)= \{ \wn \in \C^n \ | \ x_1^{k}=x_{n-(k-1)}^{n}\} . \]
\end{defn}
It is easy to see that the following "duality" holds
\beqn \label{dual}
B_n(n-k)= R_n(k) \qquad \forall k=1,\dots,n-1 \ . 
\eeqn

Finally we put
\[
m_q = \sum_{\alpha \in A} p_\alpha^q \ .
\]
Observe that $m_q^{1/q}$ is the $\L_q$-norm of the parametric vector
$\theta$.
Also we put $\rho=\max\{p_\alpha \ | \ \alpha\in A\}$, namely, the $\L_{\infty}$ norm of $\theta$.

Without lose of generality, we can think that the entries of $\theta$ are disposed in non-decreasing order, say:
$\theta = (p_1,p_2,p_3, \cdots)$, where $\rho = p_1 \geq p_2 \geq p_3 \geq \cdots$.

\section{Results}

In our main theorem we show that the cumulate distribution and the probability mass function 
of $S_n$, for strictly positive integers (namely $k\not=0$),  can be written as a geometric term 
plus a correction term.
The parameter of the geometric term is given by $m_2$.
We show also a similar result for the limiting cumulate and mass distribution functions.
Finally, we present a velocity of convergence for the convergence.

To state precisely our result we need to introduce some quantities that will appear in the theorem 
as correction terms.
The first two are related to the distribution of $S_n$.
The last two are the limits of the previous ones, 
and are related to the limiting distribution of $S_n$. 
Let
\[
a_{k,n} =
 \P\lp \bigcup_{j= \lf n/2 \rf}^{n-1} R_{n}(j)   \backslash \bigcup_{j=k}^{\lf n/2 \rf-1} R_{n}(j)\rp
+\sum_{i=k+1}^{\lf n/2 \rf -1}  \P\lp  R_{2i}(i) \backslash \bigcup_{j=k}^{i-1} R_{2i}(j) \rp    
\ ,
\]
and
$$
b_{k,n} =
 \P\lp \bigcup_{j=\lf n/2\rf}^{n-1} R_{n}(j)  \cap R_{n}(k) \backslash   \bigcup_{j=k+1}^{\lf n/2 \rf-1} R_{n}(j)  \rp
+
\sum_{i=k+1}^{\lf n/2 \rf} \P\lp R_{2i}(i) \cap R_{2i}(k) \backslash \bigcup_{j=k+1}^{i-1} R_{2i}(j)  \rp \ .
$$

Further
\[
a_k=\sum_{i=k+1}^{\infty}  \P \lp R_{2i}(i) \backslash \bigcup_{j=k}^{i-1} R_{2i}(j) \rp     \ ,
\]
and
\[
b_k =\sum_{i=k+1}^{\infty} \P\lp R_{2i}(i) \cap R_{2i}(k) \backslash \bigcup_{j=k+1}^{i-1} R_{2i}(j)  \rp \ .
\]


Now we state our main result.

\begin{teo} \label{distrib}
Let $\P$ be a product measure over $\C^{\Nset}$ with marginals identically distributed.
Then, for all positive integer $k$ and all $n \ge 2k$ 
\begin{itemize}
\item[a)] $\P(S_n\ge k)=  m_2^k + a_{k,n} \ .$
\item[b)] $\P(S_n=k) = m_2^k  - b_{k,n} \ .$
\item[c)]  $\displaystyle\lim_{n\go\infty}\P(S_n\ge k)=  m_2^k + a_{k} \ .$
\item[d)]  $\displaystyle\lim_{n\go\infty}\P(S_n=k) = m_2^k  - b_{k} \ .$
\end{itemize}
Furthermore, for all $2n \ge 4k$ one has 
$\P(S_{2n}\ge k)=\P(S_{2n+1}\ge k)$
and
$\P(S_{2n}= k)=\P(S_{2n+1}= k)$.
\end{teo}

The next corollary establishes what is the measure and the limiting measure of the 
set of strings with non-overlap, or simply the set of "self-avoiding words" \cite{RoTo}.

\begin{cor} \label{nonoverlap}
Under the hypothesis of Theorem \ref{distrib}
one has
\begin{itemize}
\item[a)] $\P(S_n = 0)=  1-m_2 - a_{1,n} \ , \ \forall n \ge 2$.

\item[b)]  
$
\displaystyle \lim_{n\go\infty}\P(S_n=0) = 1- m_2- \displaystyle\sum_{i=2}^{\infty}  \ \sum_{w \in \{S_{i}=0\} } \P(w)^2 > (1- p_1)(1-m_2) \ . 
$

\end{itemize}
Furthermore, the sequence $(\P(S_{2n} = 0))_{n \in \N}$ is decreasing.
More precisely
\[
\P(S_{2n}=0) = \P(S_{2n-2}=0) - \sum_{w \in \{S_{n}=0\}} \P(w)^2 \ .
\]
\end{cor}

\begin{rem} \label{selfavoid}
By the last statement of Theorem \ref{distrib},
$\P(S_{2n+1} = 0)=\P(S_{2n} = 0)$ for all $n$. 

\end{rem}

The next theorem provides the exponential rate of convergence of our main theorem.

\begin{teo} \label{bounds}
For every non-negative integer $k$ and every positive integer $n \ge 4k$ the following inequalities hold 
\begin{itemize}
\item[a)]
$
\lv \P(S_n = k)   - \lim_{n\go\infty}\P(S_n=k) \rv 
\le C m_2^{n/2} \lp \frac{m_3}{m_2^{3/2}} \rp^k \ ,
$
\item[b)]
$
\lv \P(S_n \ge k)   - \lim_{n\go\infty}\P(S_n \ge k) \rv 
\le C m_2^{n/2} \lp \frac{m_3}{m_2^{3/2}} \rp^k \frac{m_2^{3/2}}{m_3-m_2^{3/2}}
\ ,
$
\\
\end{itemize}
where $C$ is a positive constant (that depends only on vector $\theta$).
\end{teo}

The next proposition presents bounds for $a_{k,n},b_{k,n},a_{k},b_{k}$.

\begin{prop} \label{secondorder}
Under the hipothesis of Theorem \ref{distrib}
one has
\begin{itemize}
\item[a)] $a_{k,n} \le (m_2^{k+1}-m_2^{n})/(1-m_2)\ .$
\item[b)] $b_{k,n} 
\le 
(m_4^{k+1})/(1-m_2) + \lp 2m_2^{n/2+1}/(m_2-\rho^2)\rp   (m_3/m_2^{3/2}) ^k \ .$

\item[c)] $a_{k} \le m_2^{k+1}/(1-m_2)\ .$
\item[d)] $b_{k} 
\le 
 m_4^{k+1}/(1-m_2) \ . $

\end{itemize}
\end{prop}

\vskip.5cm
The bounds in the proposition above
do not establishes which one is the leading term 
between $m_2^k$ and $a_{k,n}$ or $a_{k}$
in Theorem \ref{distrib}.
The next proposition shows that actually, both situations can happen.
(It is obvious that $m_2^k \ge \max_{n \ge 2k}\{b_{k,n}, b_k\}$.)

The next proposition shows us that the bound presented in  Proposition \ref{secondorder}$c)$ is sharp. Moreover, it shows that, if $m_2 \leq 1/2$, $ (m_2^k)_{k\in \N}$ is the leading term. If $m_2 > 1/2$, the sequence $(m_2^k)_{k\in \N}$ starts above the sequence $(a_k)_{k\in \N}$, and then its tail becomes strictly smaller.
 
\begin{prop} \label{leadterm}
Under the conditions of heorem \ref{distrib}, there exists $A(k)$(that satisfies: $\lim_{k \to \infty}A(k) = 0$) such that $a(k)\geq m_2^{k+1}/ (1-m_2) - A(k)$. Furthermore \\
\begin{itemize}
\item[a)] If $m_2 \leq 1/2$, then $m_k > a_k$ for all $k\in \N$.
\item[b)] If $m_2 > 1/2$, then 
\begin{enumerate}
\item $a_1 < m_2$
\item There exists some $k_0 > 0$, for wich $a_k > m_2^k$, for all $k>k_0$.
\end{enumerate}
\end{itemize}
\end{prop}

\subsection{Examples}

To exemplify the behavior of 
$m_2, a_{k,n}, b_{k,n}, a_{k}, b_{k}$
we present some examples.

\begin{ex}[Two letters alphabet]
Consider the case where $\C=\{0,1\}$ and $\theta=(p,1-p) = (p_1,p_2)$, where $p_1=\rho$.
Then, $m_i= p_1^i+p_2^i$, for $i \in \N$, and the inequalities given by Proposition \ref{secondorder}
become

\begin{description}

\item[a)] $a_{k,n} \leq  \dfrac{(p_1^2 +p_2^2)^{k+1}+ (p_1^2 + p_2^2)^n}{1-p_1^2-p_2^2}$

\item[b)] $b_{k,n} \leq  \dfrac{(p_1^4 + p_2^4)^{k+1}}{1-p_1^2-p_2^2} + \dfrac{2(p_1^2 + p_2^2)^{n+1}}{p_2^2} \lp \dfrac{p_1^3+p_2^3}{(p_1^2+p_2^2)^{3/2}} \rp ^k  .$

\item[c)] $a_k \leq \dfrac{(p_1^2 +p_2^2)^{k+1}}{1-p_1^2-p_2^2}.$

\item[d)] $b_k \leq \dfrac{(p_1^4 + p_2^4)^{k+1}}{1-p_1^2-p_2^2}$.

In $c)$, notice that if $p > 1/2$, then $m_2>1/2$, and item $b)2.$ of Proposition \ref{leadterm} ($a_k > m_2^k$) holds for all $k>k_0$, where $k_0 = |\log{\frac{(1-p_1^4-p_2^4)^2}{2(p_1^2+p_2^2)-1}}| / |\log{\frac{p_1^4+p_2^4}{p_1^2+p_2^2}}|$.

\end{description}
\end{ex}

\begin{ex}[Uniform measure] \label{unif}
In this example we consider a uniform product measure 
over the finite alphabet  $\C=\{1,\dots,s\}$, so that $\theta=(1/s,\dots,1/s)$.
Then, $m_i= s \left(1/s^i\right)= 1/s^{i-1}$. Thus $m_i=1/s^{i-1}$.
The inequalities given by Proposition  \ref{secondorder}
become
\[
a_{k,n} \le \dfrac{s^{n-k}}{s^n(s-1)} , \qquad a_k \leq \dfrac{1}{s^k(s-1)}
\]

$$
b_{k,n} \le \dfrac{1}{s-1}\lp s^{-(3k+2)} + 2s^{-\frac{n+k}{2}+1}\rp , \qquad b_k \le \dfrac{s^{-(3k+2)}}{s-1}
$$

By Proposition \ref{leadterm}$a)$, we have that in the uniform case, $m_2$ is always the leading term.
\vskip 0.3cm

The proportion of words of lenght $n$ with no overlap is
\beqn \label{nonov}
\frac{s-1}{s} - 
\P\lp \bigcup_{j=n/2}^{n-1} R_{n}(j)  \backslash \bigcup_{j=1}^{n/2-1} R_{n}(j)\rp
-
\sum_{i=2}^{n/2-1}  \P\lp  R_{2i}(i) \backslash \bigcup_{j=1}^{i-1} R_{2i}(j) \rp  
\ .
\eeqn
Further
\[
\bigcup_{j=n/2}^{n-1} R_{n}(j)  \backslash \bigcup_{j=1}^{n/2-1} R_{n}(j)
=
\bigcup_{j=1}^{n-1} R_{n}(j)  \backslash \bigcup_{j=1}^{n/2-1} R_{n}(j) \ .
\]
By Lemma \ref{small}
\[
\bigcup_{j=1}^{n-1} R_{n}(j) = \bigcup_{j=1}^{n/2} R_{n}(j) \ . 
\]
Thus the leftmost probability in (\ref{nonov}) is
\[
\P\lp  R_{n}(n/2)  \backslash \bigcup_{j=1}^{n/2-1} R_{n}(j)\rp \  ,
\]
that can be added to the rightmost term in (\ref{nonov}). Thus
\[
\P(S_n=0) = 
\frac{s-1}{s} - 
-
\sum_{i=2}^{n/2}  \P\lp  R_{2i}(i) \backslash \bigcup_{j=1}^{i-1} R_{2i}(j) \rp  
\ .
\]

Similarly, the limiting proportion of words  with no overlap is exactly
\beqn \label{duplas}
\frac{s-1}{s} - 
\sum_{i=2}^{\infty}   \frac{\#\{S_{i}=0\} }{s^{2i}} 
=
\frac{s-1}{s} - 
\sum_{i=2}^{\infty}   \frac{1}{s^{i}} \P(S_{i}=0) 
 \ .
\eeqn
Since $\P(S_{2i+1}=0)=\P(S_{2i}=0)$, the last expression becomes
\[
\frac{s-1}{s} - 
2 \frac{s+1}{s} \sum_{i=1}^{\infty}   \frac{1}{s^{i}} \P(S_{2i}=0) \ . 
\]

Moreover
\[
\P(S_{2n}=0) = \P(S_{2n-2}=0) - \frac{1}{s^n}\P(S_{n}=0) \ .
\]

\end{ex}

\section{Tools for the proofs}

Before proving our main theorem, we prove a number of preparatory lemmas.
Firstly, we recall the following  classical notation. For a positive integer $x$ we write
 $\lf x \rf$ for the largest integer smaller or equal than $x$. 
Similarly, we write $\lr x \rr$ for the smallest integer larger or equal than $x$.

\begin{lem} \label{lqineq}
Let $p \ge 1$, $q \ge 1$. Then 
\[
m_{qp} \le m_q^p \ .
\]
\end{lem}

\proof 
Since $m_q^{1/q}$ is the $\L_q$ norm of the vector $\theta$, a classical 
$\L_q$ inequality gives
\[
m_{qp} = (m_{qp}^{1/qp})^{qp}\le (m_{q}^{1/q})^{qp} = m_q^p \ .
\]
\qed

The following lemma is a tool to present explicit computations for the
probability $\P(B_n(j))$.

\begin{lem} \label{newton}
The following equality holds for every positive integers $j$ and $\ell$
\[
\sum_{w \in \C^{j} } \P(w)^{\ell} = m_\ell^j \ .
\]
\end{lem} 

\proof
For each $w \in \C^{j}$ one has
\[
\P(w)= \prod_{\alpha \in A} p_\alpha^{j_\alpha} \ , 
\qquad {\rm \ where \ } \quad \sum_{ \alpha\in A} j_\alpha=j \ .
\]
Thus 
$$
\P(w)^{\ell} = \prod_{\alpha \in A} (p_\alpha^{j_\alpha})^{\ell}
             = \prod_{\alpha \in A} (p_\alpha^{\ell})^{j_\alpha} \ .
$$
Thus
\[
\sum_{w \in \C^{j} } \P(w)^{\ell}
=\sum_{ \sum_{\alpha\in A}{j_\alpha}=j} 
{j \choose \prod j_\alpha}   \prod_{\alpha \in A} (p_\alpha^{\ell})^{j_\alpha}
=\sum_{\alpha \in A} p_\alpha^{\ell} = m^j_\ell \ .
\]
\qed

The next lemma says that, the total measure of the $n$-strings 
that have small overlap remainds the same if we  "cut" the central
letters of the strings.

\begin{lem} \label{shrink}
Let $k \le \lf n/2 \rf -1$. Then 
\[
\P\lp \bigcup_{j=k}^{\lf n/2 \rf -1} R_n(j) \rp = \P\lp \bigcup_{j=k}^{\lf n/2 \rf -1} R_{2 (\lf n/2 \rf -1) }(j) \rp  .
\]
\end{lem}
\proof
$w=\wn \in \bigcup_{j=k}^{\lf n/2\rf-1} R_n(j)$ if and only if there exists a $j$ such that 
$k \le j \le  \lf n/2\rf-1$ and $\wn \in R_n(j)$.
Thus
\[
w= w_1w_2w_1  \ ,
\]
where $w_1$ is a $j$-string and $w_2$ is an $n-2j$-string and they are independent.
Now we write $w_2= w_{2,1}w_{2,2}w_{2,3}$
where $w_{2,2}$ is the central word of $w_2$, of length 2 in the case that $n$ is even or of length 3 in
the case that $w_2$ it is odd. Namely 
\[ w_{2,2} = x_{\lf n/2 \rf}^{\lr n/2 \rr+1} \ ,
\]
and $w_{2,1}$ and $w_{2,3}$ are words of length $\lf (n-2j)/2 \rf -1$.
Now, define $\tilde{w}=w_1w_{2,1}w_{2,3}w_1 \in R_{2\lf n/2\rf-1}(j)$, which is independent of $w_{2,2}$.
Thus
\beqa
\P\lp \bigcup_{j=k}^{\lf n/2 \rf-1 } R_n(j) \rp 
&=&
\sum_{w \in \cup_{j=k}^{\lf n/2 \rf-1 } R_n(j)} \P(w) \\
&=&
\sum_{w_{1}w_{2,1} \in \C^{n-2} } 
\sum_{w_{2,2}      \in \C^{i} }  \P(w_1w_{2,1}w_{2,1}w_1)\P(w_{2,2}) \ .
\eeqa
Summing independently each term, the first term sums up to $R_{2\lf n/2\rf-1}(j)$
and the second one sum up to one.
\qed

The next lemma says that the total measure of the set of $n$-strings with large overlap, goes to zero exponentially fast.

\begin{lem} \label{small}
The following holds 
\[
\P\lp \bigcup_{j=1}^{\lr n/2 \rr} B_n(j) \rp =
\P\lp \bigcup_{j= \lf n/2 \rf}^{n-1} R_n(j) \rp \le
\frac{n}{2} m_2^{\lf n/2\rf}\ .
\]
\end{lem}
\proof
The equality follows by duality.
To prove the inequality, 
firstly we have
\beqn \label{cortos}
\P\lp \bigcup_{j=1}^{\lr n/2 \rr} B_n(j) \rp 
\le
\sum_{j= 1}^{\lr n/2 \rr} \P(B_n(j)) \ .
\eeqn
Still, if $w\in B_n(j)$, then we can write $n=j\lf n/j \rf+r$ where $0 \le r < j$.
Thus 
\[
w=\underbrace{w_jw_j...w_j}_{\lf n/j \rf \ \rm{times}}w_r  \ ; \qquad w_j \in \C^j, \ w_r \in \C^r .
\]
Therefore, by Lemma \ref{newton}
\[
\P(B_n(j)) \le \sum_{w_j \in \C^{j}} \P(w_j)^{\lf n/j \rf} \rho^r
= m_{\lf n/j \rf}^{j}  \rho^r
\ .
\]
By Lemma \ref{lqineq}
$$
m_{\lf n/j \rf}^{j} 
\le
m_{2}^{(n-r)/2} \ .
$$
Observe that $\rho \le m_2^{1/2}$.
Thus, the sum in (\ref{cortos}) is bounded from above by
\[
\sum_{j= 1}^{\lr n/2 \rr} (m_2^{1/2})^{n} = 
\lf\frac{n}{2}\rf m_2^{n/2} \ .
\]
\qed



\begin{lem} \label{blocks2}
The following holds 
\[
 \bigcup_{k=1}^{n-1} R_n(k) = 
 \bigcup_{k=1}^{k=\lr n/2 \rr} R_n(k) .
\]
\end{lem}

\proof Let $k \geq \frac{n}{2}$.
If $\omega \in R_n(k) = B_n(n-k)$, so $\omega = \omega_1 \omega_1 \cdots \omega_r$, with $n =\left\lf \dfrac{n}{n-k}j \right\rf + r$, with $r$ being the size of $\omega_r$(wich could be $0$), for some integer $j$. If $r=0$, we have that $\omega$ overlaps in(at least) a $\omega_1$ string. If $r>0$, we have that $\omega$ overlaps in(at least) a $\omega_r$ string.
In both cases, we have a smaller overlap than $n/2$, and it proves that $\bigcup_{k=1}^{n-1} R_n(k) \in \bigcup_{k=1}^{k=\lr n/2 \rr} R_n(k)$, and this concludes the proof.    
\qed

\section{Proofs}

\subsection{Proof of Theorem \ref{distrib}}

\noindent\emph{ Proof of Theorem \ref{distrib}.}
For short hand notation put
\beqa
G_{n}(k)= \P(S_n \ge k) =
\P\lp \bigcup_{j=k}^{n-1} R_{n}(j) \rp \ .
\eeqa

We first consider the case when $n$ is even. 
By a simple decomposition
\beqa
G_{n}(k)&=&
\P\lp \bigcup_{j=k}^{n-1} R_{n}(j) \rp \\
&=& \P\lp \bigcup_{j=k}^{n/2-1} R_{n}(j) \rp+ \P\lp \bigcup_{j=n/2}^{n-1} R_{n}(j) \backslash
   \bigcup_{j=k}^{n/2-1} R_{n}(j)  \rp
\eeqa
By Lemma \ref{shrink}
the left most term in the last expression is equal to
\[
\P\lp \bigcup_{j=k}^{n/2-1} R_{n-2}(j) \rp \ .
\]
We can rewrite the last probability as
\[
  \P\lp \bigcup_{j=k}^{(n-2)-1} R_{n-2}(j) \rp 
 -\P\lp \bigcup_{j=n/2}^{(n-2)-1} R_{n-2}(j) \backslash
  \bigcup_{j=k}^{n/2-1} R_{n-2}(j)  \rp .
\]
The  left most term, by defintion, is equal to  
$
G_{n-2}(k) \ .
$
Thus, we conclude that
\beqa
G_{n}(k) &=&G_{n-2}(k) \\
&&+ \P(\cup_{j=n/2}^{n-1} R_{n}(j)  \backslash \cup_{j=k}^{n/2-1} R_{n}(j) ) \\
&&- \P(\cup_{j=n/2}^{(n-2)-1} R_{n-2}(j) \backslash \cup_{j=k}^{n/2-1} R_{n-2}(j) ) .
\eeqa
A similar argument shows that
\beqa
G_{n-2}(k) &=& G_{n-4}(k) \\
&&+ \P(\cup_{j=n/2-1}^{(n-2)-1} R_{n-2}(j)  \backslash \cup_{j=k}^{n/2-2} R_{n-2}(j)  ) \\
&&- \P(\cup_{j=n/2-1}^{(n-4)-1} R_{n-4}(j)  \backslash \cup_{j=k}^{n/2-2} R_{n-4}(j) ) .
\eeqa

Thus
\beqa
G_{n}(k) &=& G_{n-4}(k) \\
&&+ \P(\cup_{j=n/2}^{n-1} R_{n}(j) \backslash \cup_{j=k}^{n/2-1} R_{n}(j)  ) \\
&&- \P(\cup_{j=n/2}^{(n-2)-1} R_{n-2}(j)  \backslash \cup_{j=k}^{n/2-1} R_{n-2}(j) ) \\
&&+ \P(\cup_{j=n/2-1}^{(n-2)-1} R_{n-2}(j) \backslash \cup_{j=k}^{n/2-2} R_{n-2}(j) ) \\
&&- \P(\cup_{j=n/2-1}^{(n-4)-1} R_{n-4}(j) \backslash \cup_{j=k}^{n/2-2} R_{n-4}(j) ) .
\eeqa
Solving the two lines in between we get that they are equal to
\[
 \P( R_{n-2}(n/2-1) \backslash \cup_{j=k}^{n/2-2} R_{n-2}(j) ) \ .
\] 
A recursive
argument up to $k$ gives
\beqa
G_{n}(k) &=& G_{2k}(k) \\
&&+ \P(\cup_{j=n/2}^{n-1} R_{n}(j))  \backslash   \cup_{j=k}^{n/2-1} R_{n}(j)  ) \\
&&+ \sum_{i=k+1}^{n/2-1}  \P( R_{2i}(i) \backslash \cup_{j=k}^{i-1} R_{2i}(j) )      \\
&&- \P(\cup_{j=k+1}^{2k-1} R_{2k}(j) \backslash  \cup_{j=k}^{k}  R_{2k}(j)  ) .
\eeqa
Computing the first and last term on the right hand side of the above equality, it gives
\beqn \label{prob}
 G_{n}(k)= 
 R_{2k}(k)
+\P\lp \bigcup_{j=n/2}^{n-1} R_{n}(j)   \backslash   \bigcup_{j=k}^{n/2-1} R_{n}(j)  \rp
+\sum_{i=k+1}^{n/2-1}  \P\lp  R_{2i}(i) \backslash \bigcup_{j=k}^{i-1} R_{2i}(j) \rp    
\ .
\eeqn
This proves $a)$ since $ \P(R_{2k}(k))=m_2^k$.

Further, since the second term on the rigth hand side goes to zero as $n$ diverges, by Lemma \ref{small},
we coclude that
\beqn \label{cumulate}
\lim_{n\go\infty} G_{n}(k)= 
\P( R_{2k}(k)  )  +
\sum_{i=k+1}^{\infty}  \P( R_{2i}(i) \backslash \cup_{j=k}^{i-1} R_{2i}(j) )     \ .
\eeqn
This proves $c)$.

For the probability mass function we have
\[
\P(S_n=k)= G_n(k)-G_n(k+1) \ .
\]
And solving this equation using (\ref{prob}) we get that $\P(S_n=k)$ is equal to
\beqa
 &&\P( R_{2k}(k) ) - \P( R_{2(k+1)}(k+1)  ) \nn\\
 &&- \P\lp \bigcup_{j=n/2}^{n-1} R_{n}(j)  \cap R_{n}(k) \backslash   \bigcup_{j=k+1}^{n/2-1} R_{n}(j)  \rp\\
 &&- \sum_{i=k+2}^{n/2} \P\lp R_{2i}(i) \cap R_{2i}(k) \backslash \bigcup_{j=k+1}^{i-1} R_{2i}(j)  \rp \nn\\
 &&+ \P(R_{2(k+1)}(k+1)  \backslash   \bigcup_{j=k}^{k} R_{2(k+1)}(j)   ) \nn \ .
\eeqa
Computing the right most term in the first line with the last line in the above display, the result is
$$
-\P(R_{2(k+1)}(k+1) \cap R_{2(k+1)}(k)) \ .
$$
Considering, with some abuse of notation, that the union running over an empty set of indexes
is the empty set, we finally get that 
\beqan
\P(S_n=k) &=& \P( R_{2k}(k) )  \nn\\
 &&- \P\lp \bigcup_{j=n/2}^{n-1} R_{n}(j)  \cap R_{n}(k) \backslash   \bigcup_{j=k+1}^{n/2-1} R_{n}(j)  \rp
    \label{union}\\
 &&- \sum_{i=k+1}^{n/2} \P\lp R_{2i}(i) \cap R_{2i}(k) \backslash \bigcup_{j=k+1}^{i-1} R_{2i}(j)  \rp \ .\nn
 \eeqan
This shows $b)$.

By Lemma \ref{small}, term (\ref{union}) goes to zero as $n$ diverges.
Thus, the limit 
\[  \lim_{n\go\infty} \P(S_n=k) \ ,
\]
exists and is equal to
\beqn \label{fdp}
 \P( R_{2k}(k) )
- \sum_{i=k+1}^{\infty} \P\lp R_{2i}(i) \cap R_{2i}(k) \backslash \bigcup_{j=k+1}^{i-1} R_{2i}(j)  \rp \ .
\eeqn
This shows $d)$.

If $n$ is odd, the above argument changing $n/2$ by $\lf n/2 \rf$ holds.
We conclude that for any positive integer $n$ we have
$G_{2n+1}(k)=G_{2n}(k)$ and $\P(S_{2n+1}=k)=\P(S_{2n}=k)$.

\vskip0.5cm
\qed

\subsection{ Proof of Corollary \ref{nonoverlap}.}

Note that 
$$
\P(S_n=0)= 1- \P(S_n \ge 1) = 1-m_2 - a_{1,n} \ ,
$$ 
and similarly 
$$
\lim_{n\go\infty}\P(S_n=0)=  1-m_2 - a_{1} \ .
$$ 
But
\beq
a_1 
=
\sum_{i=2}^{\infty}  \P( R_{2i}(i) \backslash \cup_{j=1}^{i-1} R_{2i}(j) )  \ .
\eeq
The set in the probability in each term is the set
of words $ww$ where $w$ is an $i$-word without any self-overlap. Namely
\[
R_{2i}(i) \backslash \bigcup_{j=1}^{i-1} R_{2i}(j)
= \ll ww \ | \ w \in \{S_i=0\}\rl \ .
\]
This establishes the equality in  $b)$.

Now we prove the last formula.
In what follows, the first inequality is just by definition, second one  is by Lemma \ref{blocks2}
and the third one is  a simple decomposition.
\beqa
G_{n}(1)
&=&
\P\lp \bigcup_{j=1}^{n-1} R_{n}(j) \rp \\
&=&
\P\lp \bigcup_{j=1}^{n/2} R_{n}(j) \rp \\
&=& 
\P\lp \bigcup_{j=1}^{n/2-1} R_{n}(j) \rp+ \P\lp  R_{n}(n/2) \backslash
   \bigcup_{j=1}^{n/2-1} R_{n}(j)  \rp \ .
\eeqa
By Lemma \ref{shrink}, the leftmost term in the last display is equal to
$\P\lp \cup_{j=1}^{n/2-1} R_{n-2}(j) \rp$. But applying again Lemma \ref{blocks2},
this last probability equals to
$\P\lp \cup_{j=1}^{(n-2)-1} R_{n-2}(j) \rp$, which is $G_{n-2}(1)$. 

It is straightforward to see that
\[
\P\lp  R_{n}(n/2) \backslash \bigcup_{j=1}^{n/2-1} R_{n}(j) \rp
= \sum_{w \in \{S_{n/2}=0\}} \P(w)^2 \ .
\]
Thus we conclude that
\[
\P(S_{n}=0) = \P(S_{n-2}=0) - \sum_{w \in \{S_{n/2}=0\}} \P(w)^2 \ .
\]

It remains to show the strict inequality in $b)$.
By the above argument, the probability of the set of $n$-strings with some overlap
is increasing on $n$. 
Further, the above displays shows that 
\[
\P(S_{n}\ge 1) = \P(S_{n-2} \ge 1) + \sum_{w \in \{S_{n/2}=0\}} \P(w)^2 \ .
\]
Now call $p_1$ and $p_2$ the two largest $p_\alpha$, with $\alpha \in A$
(allowing multiplicities among the $p_\alpha$, tht is  $A$ is considered a multi-set, 
thus it may happen that $p_1=p_2$).
That is
$p_1= \max\{p_\alpha \ | \ \alpha \in A\}$
and $p_2= \max\{p_\alpha \ | \ \alpha \in A\backslash \alpha_0\}$ where $p_{\alpha_0}=p_1$.
It follows that, if $w \in \{S_n=0\}$ then
$\P(w)\le p_1^{n-1}p_2$.
Thus, it follows from the last display that
\[
\P(S_{n}\ge 1) \le \P(S_{n-2} \ge 1) + \P(S_{n/2}=0) p_1^{n/2-1}p_2 \ .
\]
Since $\P(S_{n}=0) $ is decreasing, 
\[
\P(S_{n}\ge 1) \le \P(S_{n-2} \ge 1) + \P(S_{2}=0) \ p_1^{n/2-1}p_2 \ .
\]
An iterative argument shows that
\[
\P(S_{n}\ge 1) \le \P(S_{2}=1) + \P(S_{2}=0) \sum_{j=1}^{n/2-1}  \   p_1^{j}p_2 \ .
\]
And 
\[
\lim_{n \go \infty}\P(S_{n}\ge 1) \le \P(S_{2}=1) + \P(S_{2}=0)  \frac{p_1p_2}{1-p_1}\ .
\]
Since $p_2 \le 1-p_1$ we conclude that
\[
\lim_{n \go \infty}\P(S_{n}\ge 1) \le \P(S_{2}=1) + p_1 \P(S_{2}=0)  \ ,
\]
observing that $\P(S_{2}=1)=m_2$.

\subsection{ Proof of Theorem \ref{bounds}. }
It follows by Theorem \ref{distrib} that 
\[
|\P(S_n = k)-\lim_{n\go\infty}\P(S_n = k)|=
|b_{k,n}-b_k| \ ,
\]
which is bounded from above by
\beqn \label{bk}
\max\ll \sum_{i=n/2+1}^{\infty} \P\lp R_{2i}(i) \cap R_{2i}(k) \rp 
,  \P\lp \bigcup_{j=n/2}^{n-1}   R_{n}(j) \cap R_{n}(k) 
         \backslash \bigcup_{j=k+1}^{n/2-1}   R_{n}(j) \rp   \rl \ .
\eeqn


\[
\]
Consider firstly the first term in (\ref{bk}).
If an $n$-string $w$ belongs to $R_{2i}(i) \cap R_{2i}(k)$ then it has the form
\[ 
w = w_{1}w_{2}w_{1}w_{1}w_{2}w_{1} \ ,
\]
where $w_1$ is a $k$-string and $w_2$ is an $i-2k$-string.
Therefore 
\[
\P(w) = \P(w_{1})^4 \P(w_{2})^2 \ ,
\]
and thus
\[
\P\lp  R_{2i}(i) \cap R_{2i}(k) \rp 
= \sum_{w_1 \in \C^{k}} \P(w_{1})^4 \sum_{w_2 \in \C^{i-2k}} \P(w_{2})^2 
= m_4^k m_2^{i-2k}\ .
\]
Summing over $i$, we get that  the first term in (\ref{bk}) is bounded by
\beqn \label{B1}
\lp\frac{m_4}{m_2^2}\rp^k  \frac{m_2^{n/2 +1}}{1-m_2}  \ .
\eeqn

Consider now the second term in (\ref{bk}).
By duality, the probability  in $b)$ is equivalent to
\beqn \label{trans}
\P\lp \bigcup^{n/2}_{j=1} B_{n}(j)  \cap R_{n}(k) \backslash   \bigcup^{n-k-1}_{j=n/2+1} B_{n}(j)  \rp
\ . 
\eeqn
Since, by definition, $B_n(j) \subset B_n(l)$ for all $l$ multiple of $j$
one has 
\[
\bigcup^{n-k-1}_{j=n/2+1} B_{n}(j) 
=
\bigcup_{j=1}^{n/2 -k} B_{n}(j) 
\ \cup \
\bigcup^{n-k-1}_{j=n/2+1} B_{n}(j) 
\ .
\]
Thus, the set in (\ref{trans}) is equal to  
\[
 \bigcup_{j= n/2-k+1}^{n/2}  B_n(j) \cap R_n(k) 
\backslash \bigcup^{n-k-1}_{j=n/2+1} B_{n}(j) 
\ .
\]


\vskip.3cm
The above expression implies that it is enough to bound 
\[
\lp \sum_{j=n/2-k+1}^{n/2-k/2} 
+
\sum_{j=n/2-k/2+1}^{n/2} \rp
\P( B_n(j) \cap R_n(k) ) 
=I+II\ . 
\]

\vskip.3cm
\vskip.3cm
Consider $I$. 
Since $2j \le n-k$ and $w \in B_n(j)$ then 
there are at least two complete blocks of length $j$ at the beginning of $w$,
and the remaining part of $w$ has length at least  $k$. Thus, we can write
$$
w=w_bw_bw_l \ .
$$ 
Further, since $w\in R_n(k)$, the first and last 
block of length $k$ are equal. Thus
$$
w= w_kw_mw_k \ .
$$
The last two descriptions of $w$ imply that 
$$
w=w_1w_2w_1w_2w_3w_1 \ .
$$
where $w_1$ has length $k$ and $w_2$ has length $j-k$.	
Moreover,  $w_3$ has length $n-2j-k$.
Thus, factorizing the measure of $w$ we have
$$
\P(w)= \P(w_1)^3 \P(w_2)^2 \P(w_3) \ .
$$
Recall that  $\rho= \max\{ p_\alpha \ | \ \alpha\in A\}$. 
Therefore
\[
\P\lp  B_{n}(j) \cap R_{n}(k) \rp 
\le\sum_{w_1\in C^k} \P(w_1)^3 \sum_{w_2\in C^{j-k}} \P(w_2)^2 \rho^{n-k-2j}
= m_3^{k} m_2^{j-k} \rho^{n-k-2j} \ .
\]
Summing over $j$ we have
\beqn \label{B2}
\sum_{j=n/2-k+1}^{(n-k)/2} \P\lp  R_{n}(j) \cap R_{n}(k) \rp   
\le C_\theta m_2^{n/2} \lp \frac{m_3}{m_2^{3/2}} \rp^k \ ,
\eeqn

where $C_{\theta} = m_2/(m_2 - \rho^2)$.
Finally, observe that ${m_3}/{m_2^{3/2}} <1$ is equivalent to 
${m_3}^{1/3} < {m_2^{1/2}}$ which is true by Lemma \ref{lqineq}.

\vskip.3cm
Consider $II$.
Take $w=x_1^n\in B_{n}(j) \cap R_{n}(k)$.
Since $w\in  R_{n}(k) $ one has 
$$
w=w_kw_mw_k \ .
$$
Since blocks can be read forward or backward, every peace of the string is also periodic (that is, the central peace is in $B_{n-2k}(j)$). So, we can recopilate this and write
$$
w=w_1w_2w_1w_2w_3w_1 \ . 
$$
The length of $w_1$ is $k$. The length of $w_2$  is $n-2k-j$
and the length of $w_3$ is  $2j+k-n$.
Factorizing the measure of $w$ we have
\beqan \label{w}
\P( B_n(j) \cap R_n(k)) 
&\le& 
\sum_{w_1 \in \C^k} \P(w_1)^3  \sum_{w_2 \in \C^{n-2k-j}} \P(w_2)^2 \rho^{2j+k-n} \\
&=& m_3^k m_2^{n-2k-j} \rho^{2j+k-n}\ .
\eeqan
Summing over $j$ we have
$$
\sum_{j=(n-k)/2 +1}^{n/2} \P\lp  R_{n}(j) \cap R_{n}(k) \rp   
\le C'_\theta  m_2^{n/2} \lp \frac{m_3}{m_2^{3/2}} \rp^k \ 
$$
where $C'_\theta = (\rho/m_2)^2$.
This ends the proof of $II$.


So, as $C_\theta \geq C'_\theta$, take $C = 2C_\theta$
To end the proof of $b)$ we need to show that the right hand side of
(\ref{B1}) is less or equal than (\ref{B2}).
To this, observe that  this is equivalent to show that
$
m_4 \le m_3 m_2^{1/2} \ .
$
But 
\[
m_4 = \sum_{\alpha\in A} p_\alpha^4 
\le \rho \sum_{\alpha\in A, } p_\alpha^3
= \rho m_3 \ ,
\]
and
\[
\rho =  (\rho^2)^{1/2} 
\le (\rho^{2} + \sum_{\alpha\in A, p_\alpha\not= \rho} p_\alpha^2)^{1/2} 
= m_2^{1/2} \ .
\]
This ends the proof of $a)$.

\vskip.5cm

The proof of $b)$ follows directly from $a)$ summing up the error terms in $a)$.

\qed

\subsection{ Proof of Proposition \ref{secondorder}.}

We first prove $c)$.
We can write
\[
a_k= \lp \sum_{i=k+1}^{2k-1}+\sum_{i=2k}^{\infty} \rp  
\P( R_{2i}(i) \backslash \cup_{j=k}^{i-1} R_{2i}(j) ) = I+II   
\  .
\]
As in the proof of the first term in (\ref{bk}) with $n=4k$ we get 
\beq
II \le \lp\frac{m_4}{m_2^2}\rp^k  \frac{m_2^{2k +1}}{1-m_2} 
= m_4^k  \frac{m_2}{1-m_2} 
 \ .
\eeq

By a direct computation one has
\[
I \le \sum_{i=k+1}^{2k-1}\P( R_{2i}(i))=\sum_{i=k+1}^{2k-1}m_2^j = \frac{m_2^{k+1}-m_2^{2k}}{1-m_2}
\ .
\]
Thus, $c)$ follows since $m_4^k \le m_2^{2k}$.

\vskip.5cm

Proof of $d)$.
\[
b_{k} \le
\lp \sum_{i=k+1}^{2k-1} +\sum_{i=2k}^{\infty} \rp
 \P\lp R_{2i}(i) \cap R_{2i}(k)   \rp 
=I+II\ .
\]
As we computed in the proof of an upper bound for (\ref{bk}) when proving
Theorem \ref{bounds}, $II$ is 
\[
m_4^k \sum_{i=0}^{\infty} m_2^i = m_4^k \frac{1}{1-m_2} \ .
\]

For the leading term in $I$ we note that
\[
 R_{2i}(i) \cap R_{2i}(k)   = \{ ww \ | \ w \in   B_{k+j}(j) \} \ ,
\]
with $j=i-k$.
Thus, for $1 \le j \le k-1$ we compute
\[
\sum_{ w \in   B_{k+j}(j) } \P(w)^2 
\le 
\sum_{ w_j \in   \C^{j} } \P(w_j)^{2 \lf (k+j)/j \rf} \rho^{2r}
= 
m_{2 \lf (k+j)/j \rf}^{j} \rho^{2r} \ .
\]
where $k+j=\lf \frac{k+j}{j} \rf j +r$ and $0 \le r \le j-1$.
We conclude that
\[
m_{2 \lf (k+j)/j \rf}^{j} \rho^{2r} \le m_4^j\rho^{2r} \le m_4^{(k+j-r)/2}\rho^{2r} 
\le m_4^{(k+j)/2} \ .
\]

Therefore
\beqan \label{leading}
I \le m_4^{k} \frac{m_4-m_4^{k}}{1-m_4} \ .
\eeqan


Proof of $a)$ and $b)$. 
Similar computations of those done in the proof of $c)$ and $d)$
can be done to get an upper bound for the second term in $a_{k,n}$ and $b_{k,n}$.

The second term in $a_{k,n}$ is bounded by 
\[
\sum_{i=k+1}^{n/2-1}  \P( R_{2i}(i)  ) 
= \sum_{i=k+1}^{n/2-1} m_2^i = \frac{m_2^{k+1}-m_2^{n/2}}{1-m_2}\ ,
\]
and the first one by
\[
\sum_{i=n/2}^{n-1}  \P( R_{n}(i)  )
=\sum_{i=n/2}^{n-1} m_2^{i} 
=\frac{m_2^{n/2}-m_2^{n}}{1-m_2}\ .
\]
Thus, $a_{k,n} \le (m_2^{k+1}-m_2^{n})/(1-m_2)$.

The first term in $b_{k,n}$ was bounded in the proof of Theorem \ref{bounds}, equation  (\ref{bk}) by 
$ C_\theta m_2^{n/2} \lp \frac{m_3}{m_2^{3/2}} \rp^k $. 
The second one 
is bounded as  was done $b_k$ above.
\qed

\subsection{ Proof of Proposition \ref{leadterm}.}

$a)$ follows directly from Proposition \ref{secondorder} $c)$.

Now we prove $b.1)$.
By Corollary \ref{nonoverlap} $b)$, we have
$a_1 < p_1(1-m_2) < 1-m_2 \le m_2$. Last inequality follows since $m_2 \ge 1/2$. 

Now we prove the first sentence and also $b.2)$.
By definition
\beqa
a_k 
&=& \sum_{i=k+1}^{\infty}  \P \lp R_{2i}(i) \backslash \bigcup_{j=k}^{i-1} R_{2i}(j) \rp \\
&=& \sum_{i=k+1}^{\infty}  \P \lp R_{2i}(i) \rp -  \sum_{i=k+1}^{\infty}(R_{2i}(i) \bigcap \bigcup_{j=k}^{i-1} R_{2i}(j)) \\
&=&
\sum_{i=k+1}^{\infty}m_2^i - \sum_{i=k+1}^{\infty}\P(R_{2i}(i) \bigcap \bigcup_{j=k}^{i-1} R_{2i}(j)) 
\eeqa
Bounding the union by the sum we get
\beqa
a_k
&\ge&\sum_{i=k+1}^{\infty}m_2^i - \sum_{i=k+1}^{\infty} \sum_{j=k}^{i-1}\P(R_{2i}(i)\bigcap R_{2i}(j))\\
&=& \cfrac{m_2^{k+1}}{1-m_2} -  \sum_{j=k}^{\infty} \sum_{i=j+1}^{\infty}\P(R_{2i}(i)\bigcap R_{2i}(j)),
\eeqa
where the equality was obtained by using Fubini's Theorem. 
Now, let's take a look at the last term on the previous equation. it can be written as
$$
\sum_{j=k}^{\infty}\sum_{i=j+1}^{2j -1}\P(R_{2i}(i)\bigcap R_{2i}(j)) + 
\sum_{j=k}^{\infty}\sum_{i=2j}^{\infty}\P(R_{2i}(i)\bigcap R_{2i}(j)) = I+ II
$$
Term $I$ is bounded as in (\ref{leading}):
\[
I \le m_4^k\lp \dfrac{m_4-m_4^k}{1 - m_4} \rp  \leq \dfrac{m_4^{k+1}}{1-m_4}.
\]

For the second one we have
\beqa
II
&=&
\sum_{j=k}^{\infty}\sum_{i=2j}^{\infty} 
\lp \sum_{\omega \in \C^j}\P(\omega)^4 \rp \lp \sum_{\omega \in \C^{i-2j}} \P(\omega)^2\rp \\
&=& \sum_{j=k}^{\infty}\sum_{i=2j}^{\infty} m_4^j m_2^{i-2j} \\
&=& \lp \dfrac{m_4^k}{1-m_4} \rp \lp \dfrac{1}{1-m_2} \rp .
\eeqa

So, we have
$$
a_k \geq m_2^k\lp\dfrac{m_2}{1-m_2} - \dfrac{A(k)}{m_2^k}\rp,
$$

where 
$$
A(k) = \dfrac{m_4^k}{1-m_4} \lp m_4 + \dfrac{1}{1-m_2} \rp \geq I+II.
$$
Clearly, $\lim_{k\to\infty}A(k)/m_2^k = 0$, and also $\lim_{k\to\infty}A(k) = 0$. Putting on the most left side the lower bound given by Theorem \ref{secondorder}, we have that
$$
\dfrac{m_2^{k+1}}{1-m_2} - A(k) \leq a_k \leq \dfrac{m_2^{k+1}}{1-m_2},
$$
and this proves sharpness. To prove $b.2)$, we just have to notice that since, by hypothesis, $m_2/(1-m_2) > 1$, then there is some $k_0$ for which: 
$$\dfrac{m_2}{1-m_2} - \dfrac{A(k)}{m_2^k} \geq 1, \forall k > k_0.$$ 
And this concludes the proof. \qed

\section{Non-convergence in probability}

In this section we show that $S_n$ does not converges in probability when $n$ goes to infinity.
Recall that since we are considering non-trivial cases, we have $\rho < 1$.

%
%

\begin{prop}
Under the conditions of Theorem \ref{distrib},
there is not a random variable $S$ over $\C^{\Nset}$ such that
$S_n$ converge in probability to $S$.
\end{prop}    

\begin{proof}
Suppose that $S_n$ converges to $S$ in probability.
Then, for all $\epsilon>0$
\[
\lim_{n \go \infty} \P(|S_{n+1}-S_n|< \epsilon) =1 \ .
\]
Consider $\epsilon<1$.
Since, by definition  $S_n=n-T_n$ one has
$$
\{ | S_{n+1} - S_n| < \epsilon \} = \{ |T_n - T_{n+1} + 1| < \epsilon \} 
\ . 
$$
Since $T_n$ is non decreasing and takes only positive integer values
\[
\{ |T_{n+1} - T_{n}| < \epsilon \}
= \{ T_{n+1} = T_{n} \} = \{ S_{n+1} = S_n+1 \}\ .
\]
Conditioning on $\{T_n=k\}$ we get
$$
\P(T_{n+1} = T_{n} ) = \sum_{\alpha \in A} p_\alpha^2 < 1\ .
$$
This ends the proof.
\end{proof}

{\bf Acknowledgments.}
We thank Anatoli Yambarstev, Andrea Vanessa Rocha and Andrei Toom who let us know about
Janson's argument.
The problem of computing the fluctuations of $S_n$ was suggested by Pablo Ferrari.
The problem of counting the number of non-selfoverlapping strings was suggested by Andrei Toom on the Brazilian School of Probability . 
M.A. is partialy supported by CNPq grant 312904/2009-6.
R. L. received CNPq grant 560764/2008-1, between 2008 and 2010.
This paper is part of the Projeto Regular Fapesp 2010/19748-7.
This work is part of USP project "Mathematics, computation, language and the brain''(Programa da Reitoria  da Universidade de S\~{a}o Paulo de Incentivo \`{a} Pesquisa - Projeto MaCLinC - Matem\'{a}tica, Computa\c{c}\~{a}o, Linguagem e C\'{e}rebro, Processo USP no. 2011.1.9367.1.5).

\end{document}